\documentclass[12pt]{amsart}
\usepackage{amssymb,amsmath}
\usepackage{geometry}    
\usepackage{mathrsfs}

\usepackage{graphicx}

\usepackage{vmargin}

\usepackage{tikz}  

\setmargrb{1in}{1in}{1in}{1in}

\newtheorem{theorem}{Theorem}[section]
\newtheorem{lemma}[theorem]{Lemma}
\newtheorem{proposition}{Proposition}
\newtheorem{corollary}[theorem]{Corollary}
\theoremstyle{definition}

\newtheorem{remark}{Remark}

\newtheorem{conjecture}{Conjecture}
\newtheorem{problem}{Problem}

\usepackage{yfonts}

\usepackage{xcolor}

\usepackage{bbm}

\begin{document}
	
	\title[Disproof of the uniform Littlewood conjecture]{Disproof of the uniform Littlewood conjecture}
	
	\author{Johannes Schleischitz}

	\thanks{ School of Computer, data and mathematical sciences (CDMS),
    Western Sydney University, Australia  \\
		J.schleischitz@westernsydney.edu.au}

\begin{abstract}
      We show that the uniform Littlewood Conjecture (ULC) recently introduced by Bandi, Fregoli and Kleinbock is false. More precisely the counterexamples form a residual set, the method further suggests positive Hausdorff dimension.
      For a mildly twisted
      problem, we indeed separately show that the Hausdorff dimension is at least $1$.
      Moreover, we disprove a uniform version of the $p$-adic Littlewood problem,
      as well as some twisted weaker version of a more general $S$-arithmetic setting, for any proper subset (possible infinite) of primes $S$. The latter contrasts the classical (non-uniform) case where the answer is known to be affirmative when $S$ has at least two elements. The disproof of ULC, our main new result, is semi-constructive; the non-constructive part involves effective 
      results on Zaremba's famous conjecture
      by Bourgain and Kontorovich, as well as estimates for the cardinality 
      of product sets over finite fields.
\end{abstract}

\maketitle

{\footnotesize{

		{\em Keywords}: multiplicative Diophantine approximation, uniform approximation, Littlewood conjecture \\
		Math Subject Classification 2020: 11J13}}

\section{ On the uniform Littlewood conjecture }  \label{intro}


    The famous Littlewood Conjecture 
    asks if for any real numbers $\xi,\zeta$ we have
    \begin{equation} \label{eq:lw}
        \liminf_{q\in \mathbb{Z}, q\to\infty} q \Vert q\xi\Vert\cdot \Vert q \zeta\Vert=0.
    \end{equation}
     Here $\Vert.\Vert$ denotes the distance to the nearest integer.
    In~\cite{bfk}, Bandi, Fregoli and Kleinbock 
    studied a uniform version.
    They showed that the set of pairs of real numbers $\xi, \zeta$ satisfying 
    \begin{equation} \label{eq:LW}
        \lim_{Q\to\infty} Q \min_{q\in \mathbb{Z}, 0<q\le Q} \Vert q\xi\Vert\cdot \Vert q \zeta\Vert=0
    \end{equation}
    still has full $2$-dimensional Lebesgue measure. In fact a more general version for $m\times n$ matrices is established.
    The question arises if there are any exceptions to \eqref{eq:LW}, or the more general claim, at all. This was called uniform Littlewood conjecture (ULC) in~\cite{bfk}. We state this reformulation of~\cite[Question 1.11]{bfk}  for $m=2, n=1$.

    \begin{conjecture}[ULC]
        For any real numbers $\xi, \zeta$ we have \eqref{eq:LW}.
    \end{conjecture}

   Another recent preprint by Kleinbock and Wu~\cite{kw}
   also addresses ULC above.
   We prove that ULC as stated above is false
   and in fact \eqref{eq:LW} only holds for a meager set, thus a small set in topological sense. For any set in Euclidean space $A\subseteq \mathbb{R}^k$
   we write $A+A=\{a+a: a\in A\}$ for the sumset, likewise for product sets $A\cdot A$ when $k=1$.

   \begin{theorem}  \label{T01}
       The ULC is false, we have
       \[
       \sup_{(\xi,\zeta)\in\mathbb{R}^2 }\;\left\{\limsup_{Q\to\infty} Q \min_{q\in \mathbb{Z},0<q\le Q} \Vert q\xi\Vert\cdot \Vert q \zeta\Vert\right\} > 0.005326 > \frac{1}{188}.
       \]
       More precisely, the set 
       \[
       \Theta:=\left\{ (\xi,\zeta)\in\mathbb{R}^2: \;\limsup_{Q\to\infty} Q \min_{q\in \mathbb{Z},0<q\le Q} \Vert q\xi\Vert\cdot \Vert q \zeta\Vert > \frac{1}{188}  \right\}
       \]
       is a dense $G_{\delta}$ set.
       Consequently $\Theta$ has full packing dimension and $\Theta+\Theta=\mathbb{R}^2$.
   \end{theorem}

     \begin{remark}  \label{remark1}
    The best constant derived unconditionally from our method is some irrational algebraic number slightly larger than the stated bound, derived from some rather tedious optimization problem. We did not put effort to find its minimal polynomial, as it is not expected to be optimal.
    \end{remark}

    The numerical constant is considerably smaller than the expected value, as discussed in more detail in~\S~\ref{below} below.
    Our proof has a semi-constructive character. The first key idea is a concrete construction of counterexamples, however some rather non-constructive arguments are used to justify that this construction is well-defined. A key ingredient for the non-constructive part is a quantitative result by Bourgain and Kontorovich~\cite{bko} on Zaremba's conjecture. Further some result on cardinality of product sets in finite fields are needed, which can be settled with Erd\H{o}s Turan inequality and some classical exponential sum estimate due to Vinogradov. The method directly yields the density of $\Theta$. The $G_{\delta}$ property can be proved separately in a rather straightforward way, the bound $1/188$ within $\Theta$ is irrelevant for this argument. The implications on packing dimension and sumsets
     in Theorem~\ref{T01} are readily deduced as follows: Proposition~\ref{erpro} below, essentially due to Erd\H{o}s~\cite{erd}, settles the sumset implication. 

       \begin{proposition}[Erd\H{o}s]  \label{erpro}
         If $A, B\subseteq \mathbb{R}^k$ are dense $G_{\delta}$ sets then $A+B=\mathbb{R}^k$,
         and if $k=1$ also $A\cdot B=\mathbb{R}$.
     \end{proposition}

       Its short proof in~\cite{erd} relies on Baire's Theorem.
      The implication regarding packing dimension
     is consequence of the
     below result from the arXiv resource~\cite[Theorem 3.3]{icharxiv}. 

    \begin{lemma}[\cite{icharxiv}] \label{icha}
        If $R\subseteq \mathbb{R}^k$ is residual (dense $G_{\delta}$) then it has full packing dimension.
    \end{lemma}

     Since~\cite{icharxiv} has never been published in a journal, we repeat the short proof involving Proposition~\ref{erpro} for convenience of the reader. Denote by $\mathcal{L}_k$ the set of Liouville vectors in $\mathbb{R}^k$, defined 
     as the set of real vectors $(\xi_1,\ldots,\xi_n)$ for which
       \begin{equation}  \label{eq:defana}
    \liminf_{q\in\mathbb{Z},\; q\to\infty} q^N \max_{1\le j\le k}  \Vert q \xi_j\Vert=0
    \end{equation}
    for any $N$.
     It is well-known that $\mathcal{L}_k$ form dense $G_{\delta}$ sets of Hausdorff dimension $0$, Oxtoby's book~\cite{oxtoby} contains short proofs, see also Lemma~\ref{gdelta} below.
     Denote by $\dim_H$ resp. $\dim_P$
     the Hausdorff resp. packing dimension of a Euclidean set.

    \begin{proof}
    Both $\mathcal{L}_k$ and $R$
    are residual, so the sumset satisfies $R+\mathcal{L}_k=\mathbb{R}^k$ by Proposition~\ref{erpro}. On the other hand, an estimate by Tricot~\cite{tricot} and the fact that $\dim_H \mathcal{L}_k=0$ 
    shows that 
    $\dim_H(R+\mathcal{L}_k)\le \dim_H \mathcal{L}_k + \dim_P R=\dim_P R$. Combination concludes the proof.
\end{proof}


     In contrast, the counterexamples to \eqref{eq:lw} form a meager set, let us call it $\mathcal{A}$. Indeed, its elements are badly approximable (see \eqref{eq:FGH} below for a definition) and hence lie in the complement of $\mathcal{L}_2$, which is dense $G_{\delta}$ by Proposition~\ref{erpro}. Thus the intersection $\mathcal{A}^c\cap \Theta$ via
     Baire's Theorem together with Theorem~\ref{T01}
     provide us with a dense $G_{\delta}$ set of vectors
     that satisfy the classical Littlewood Conjecture \eqref{eq:lw} but not the uniform version ULC. This in particular answers both parts a), b) of~\cite[Question~1.11]{bfk} in the negative, at least for $m=2, n=1$.
      In fact all
      the numbers $\xi, \zeta$ 
      forming counterexamples to ULC
      concretely constructed 
    in our proof of Theorem~\ref{T01}, thereby a subset of $\Theta$,
    both have irrationality exponent $\ge 3$; thus they certainly do not provide counterexamples for the classical Littlewood conjecture, as this would require them both to be badly approximable. 
    In this context we further point out
    that again by Baire's Theorem the set $\mathcal{L}_2\cap \Theta$ is dense $G_{\delta}$ as well. In particular the numbers $\xi,\zeta$ 
    in Theorem~\ref{T01} can be chosen
    Liouville numbers, in fact forming Liouville vectors, a stronger property.
     On the other hand, we do not know if all elements of $\Theta$ (or more generally counterexamples to ULC) exhibit good ordinary (i.e. non-uniform) approximation properties, see Problem~\ref{P3} in~\S~\ref{oppro} below.


\section{A twisted problem} \label{s1.2}
In the next section we address some variant of the ULC.
It is possible to obtain metrical refinements of Theorem~\ref{T01} if we twist the right hand side of ULC by a function that tends to zero (arbitrarily slowly).
This result Corollary~\ref{tth} below, as well as some other claims, will follow from some more general going up principle Theorem~\ref{2t}.
The proofs of all results in this section are very different and considerably easier than for the main result Theorem~\ref{T01}. 

We introduce some notation.
For $\Psi: \mathbb{N}\to (0,\infty)$ any function and $m\ge 1$ an integer, define
\[
\mathcal{A}_m(\Psi):= \left\{ (\xi_1,\ldots,\xi_m)\in\mathbb{R}^m: \; 
 \liminf_{q\to\infty} \frac{q}{\Psi(q)} \prod_{i=1}^{m} \Vert q\xi_i\Vert > 0\right\},
\]
and
\[
\mathcal{B}_m(\Psi)= \left\{ (\xi_1,\ldots,\xi_m)\in\mathbb{R}^m: \; 
 \limsup_{Q\to\infty} \frac{Q}{\Psi(Q)} \min_{q\in \mathbb{Z},0<q\le Q} \prod_{i=1}^{m} \Vert q\xi_i\Vert > 0\right\}.
\]
The
sets $\mathcal{A}_m(\Psi)$ resp. $\mathcal{B}_m(\Psi)$
are essentially the complements of the classical sets of multiplicatively ordinarily resp. uniformly $\Psi(k)/k^2$ approximable real column vectors, 
so they almost correspond to $W_{m,1}^{\times}(\Psi(k)/k)^c$ resp. $D_{m,1}^{\times}(\Psi(k)/k)^c$ in notation of~\cite{bfk}.
Clearly $\mathcal{B}_m(\Psi)\supseteq \mathcal{A}_m(\Psi)$ for any $m\ge 1$. If $\mathbbm{1}$ denotes the constant $1$ function,
we have
$\mathcal{A}_2(\mathbbm{1})=\mathcal{A}$ with the latter set defined in~\S~\ref{intro}, so
the Littlewood Conjecture
\eqref{eq:lw} states that $\mathcal{A}_2(\mathbbm{1})=\emptyset$,
whereas Theorem~\ref{T01} implies that
$\mathcal{B}_2(\mathbbm{1})\supseteq \Theta$ is 
dense $G_{\delta}$, thus non-empty. 

We prove some going up inclusion theorem involving ordinary and uniform approximation sets and discuss interesting special cases below.

\begin{theorem}  \label{2t}
     Let $m\ge 2$ be an integer. Let 
     $\Psi: \mathbb{N}\to (0,\infty)$ be any function and
$\Phi: \mathbb{N}\to (0,\infty)$ be any function tending to infinity as $t\to\infty$, arbitrarily slowly.
    Then there is a dense $G_{\delta}$ set $\mathcal{G}=\mathcal{G}(\Phi)\subseteq \mathbb{R}$ such that for any function $\Psi$ we have
    \[
    \mathcal{B}_m(\Psi/\Phi)\supseteq \mathcal{A}_{m-1}(\Psi)\times \mathcal{G}.
    \]
\end{theorem}

Note that we do not require 
monotonicity of $\Psi$.
Let $Bad=\mathcal{A}_1(\mathbbm{1})$ denote the set of badly approximable real numbers. 
For $m=2$ we readily get the aforementioned metrical refinement.

\begin{corollary} \label{tth}
    Assume $\Psi(t)\to 0$ as $t\to\infty$. 
    Then there is a dense $G_{\delta}$ set $\mathcal{G}=\mathcal{G}(\Psi)\subseteq \mathbb{R}$ such that
    \begin{equation} \label{eq:Incl}
    \mathcal{B}_2(\Psi) \supseteq Bad \times \mathcal{G}.
    \end{equation}
    In particular $\mathcal{B}_2(\Psi)$ has Hausdorff dimension at least $1$ and full packing dimension. Moreover $\mathcal{B}_2(\Psi)$ is a dense $G_{\delta}$ set, and  the sumset $\mathcal{B}_2(\Psi)+\mathcal{B}_2(\Psi)$ equals $\mathbb{R}^2$.
\end{corollary}

The main claim \eqref{eq:Incl} follows by letting $\Psi=\mathbbm{1}$ in Theorem~\ref{2t} and identifying $\Psi$ of the corollary with
$\tilde{\Psi}=\Psi/\Phi=1/\Phi$, which indeed captures any function tending to $0$.
Moreover we used the well-known fact the $Bad= \mathcal{A}_1(\mathbbm{1})$ has full Hausdorff dimension
due to Jarn\'ik~\cite{jarnik}, 
generalized in~\cite{schmidt}.
The dense $G_{\delta}$ property,
full packing dimension and sumset property  of $\mathcal{B}_2(\Psi)$ follow already a fortiori from Theorem~\ref{T01} as $\Theta\subseteq \mathcal{B}_2(\mathbbm{1})\subseteq \mathcal{B}_2(\Psi)$.
Very similarly for larger $m$ we can derive from Theorem~\ref{2t} 

\begin{corollary} \label{bewleicht}
    Let $m\ge 3$ be an integer. Assume $\Psi(t)\to 0$ as $t\to\infty$. 
     Then there is a dense $G_{\delta}$ set $\mathcal{G}=\mathcal{G}(\Psi)\subseteq \mathbb{R}$ such that
    \[
    \mathcal{B}_m(\Psi)\supseteq \mathcal{A}_{m-1}(\mathbbm{1})\times \mathcal{G}.
    \]
\end{corollary}

One interpretation of Corollary~\ref{bewleicht} is that if $\mathcal{B}_3(\Psi)$ happens to have packing dimension less than $1$ for some $\Psi\to 0$, then the Littlewood Conjecture~\eqref{eq:lw} is true, in view of Lemma~\ref{icha} below. However, we believe the packing dimension is full.
Note that for $m=2$ the claim simplifies
to Corollary~\ref{tth}.
Clearly if the Littlewood conjecture \eqref{eq:lw} is true
then Corollary~\ref{bewleicht} is pointless. Another corollary reads as follows.

\begin{corollary} \label{lastc}
  Suppose $\Psi:\mathbb{N}\to (0,\infty)$ is any function such that
    \begin{equation} \label{eq:converges}
    \sum_{k=1}^{\infty} \frac{\Psi(k)}{k}<\infty.
    \end{equation}
    Then there exists some full Lebesgue measure set $\Omega\subseteq \mathbb{R}$ and some dense $G_{\delta}$ set $\mathcal{G}\subseteq \mathbb{R}$ so that
    $\mathcal{B}_2(\Psi)\supseteq \Omega\times \mathcal{G}$,
    or equivalently for any $\xi\in \mathcal{G}$ and $\zeta\in \Omega$
     \begin{equation*} 
    \limsup_{Q\to\infty} \frac{Q}{\Psi(Q)} \min_{q\in \mathbb{Z},0<q\le Q} \Vert q\xi\Vert\cdot \Vert q \zeta\Vert > 0.
    \end{equation*}
    Thus, by symmetry, the projection of $\mathcal{B}_2(\Psi)$ onto the coordinate axes has full measure.
\end{corollary}

\begin{proof}
We use the following elementary observation.

    \begin{proposition}
        Let $\sum a_k<\infty$ be a convergent series
        of real numbers $a_k>0$. Then there exists a sequence $b_k\to\infty$ such that 
        $\sum a_k b_k<\infty$.
    \end{proposition}

    We leave the details of the proof to the reader.
    Starting with $\Psi$, by the proposition, we find some $\varphi\to\infty$ so that 
    \begin{equation} \label{eq:125}
        \sum_{k=1}^{\infty} \Psi(k) \varphi(k)/k<\infty
    \end{equation}
    by choosing $a_k=\Psi(k)/k$ and defining $\varphi$ via $\varphi(k):=b_k$. We apply 
     the case $m=2$ of Theorem~\ref{2t} to
     the function $\Psi\varphi$. Then we know
     for any function $\Phi\to\infty$ we have
      \[
     \mathcal{A}_{1}(\Psi\varphi)\times \mathcal{G}\subseteq
    \mathcal{B}_2(\Psi\varphi/\Phi).
     \]
    Choosing $\Phi=\varphi$
    the right hand side 
    becomes $\mathcal{B}_2(\Psi)$.
    The set $\Omega=\mathcal{A}_1(\Psi\varphi)$ has full measure by
    \eqref{eq:125} and Jarn\'ik Besicovich Theorem (we only need the straightforward convergence part, which requires no monotonicity on $\Psi\varphi$), so the claim follows.
\end{proof}


It is possible to generalize the claim 
to arbitrary $m$, then according to Gallagher's Theorem condition \eqref{eq:converges}
has to be replaced by $\sum_{k=1}^{\infty} \Psi(k) (\log k)^{m-1}/k<\infty$. We state two more remarks on Corollary~\ref{lastc}.

\begin{remark}
    In the final metrical projection claim, it follows in fact easily from  $\mathcal{B}_2(\Psi)\supseteq \Omega\times \mathcal{G}$
    and symmetry of $\mathcal{B}_2(\Psi)$ in $\xi,\zeta$ that its projection onto any line in $\mathbb{R}^2$ has full measure (with respect to the natural Haar measure on the line).
\end{remark}

\begin{remark}
    The final metrical projection claim
    follows in many interesting instances alternatively 
    from~\cite[Theorem 1.9]{bfk}.
    Indeed if $\Psi(k)=o(1/\log k)$,
    then $\mathcal{B}_2(\Psi)\subseteq \mathbb{R}^2$ has full $2$-dimensional Lebesgue measure by~\cite[Remark 1.10, (ii)]{bfk}, implying that the projections onto axes (more generally onto any line) have full $1$-dimensional Lebesgue measure.
    This also contains certain $\Psi$ with divergent series $\sum \Psi(k)/k$ not covered by Theorem~\ref{2t}.
    On the other hand, it is not hard to construct (even decreasing) functions $\Psi$ for which the assumption
    \eqref{eq:converges} of Theorem~\ref{2t} holds but which do not satisfy $\Psi(k)=o(1/\log k)$.  Moreover we emphasize that a strength of 
    the more general claim $\mathcal{B}_2(\Psi)\supseteq \Omega\times \mathcal{G}$ is that the set $\mathcal{G}$ is independent from $\zeta\in \Omega$; the existence of a non-empty $\mathcal{G}$ with this property for some full measure $\Omega\subseteq \mathbb{R}$ does not follow from full measure of $\mathcal{B}_2(\Psi)\subseteq \mathbb{R}^2$.
\end{remark}

    \section{$S$-arithmetic ULC}

A uniform version 
of the well-known $p$-adic Littlewood conjecture would naturally ask if
\[
\lim_{Q\to\infty} Q  \min_{q\in \mathbb{Z},0<q\le Q} \Vert q\xi\Vert \cdot \vert q \vert_{p}
   =0
\]
for every real number $\xi$. We show that
this is false as well. We further consider
$S$ arithmetic settings, where we allow also infinite sets $S$ of primes. We show
that for any reasonable set $S$,
a strengthening of this claim in the spirit 
of~\S~\ref{s1.2}, i.e.
with an additional factor that tends arbitrary slowly to infinity, also does not hold. 
On the other hand, the classical $S$ arithmetic version is false when $S^c$ is only a singleton.

We recall the definition: 
Let $S$ be set of primes. For $q\ne 0$ 
an integer
let $|q|_S= \prod_{p\in S} |q|_{p}$ where $|q|_p= p^{-t}$ if $p^t\Vert q$ is the $p$-adic evaluation of $q$. Note this is also defined for infinite $S$ as $|q|_p=1$ for all but finitely many $p\in S$. When we write $S^c$ below,
we always mean the complement is taken within the set of primes.
We prove the following dichotomy result.

\begin{theorem} \label{t2}
Let $S$ be any proper subset of primes (possibly infinite). Then
\begin{itemize}
\item[(i)] In the $p$-adic setting, i.e. if $|S|=1$, the set
\[
 \Gamma:= \left\{ \xi\in\mathbb{R}:\;\limsup_{Q\to\infty} Q \min_{q\in \mathbb{Z},0<q\le Q} \Vert q\xi\Vert \cdot \vert q \vert_{S}
     = 1\right\}
\]
is dense $G_{\delta}$ and has Hausdorff dimension at least $1/2$, full packing dimension and $\Gamma+\Gamma=\Gamma\cdot \Gamma=\mathbb{R}$. 
    \item[(ii)] 
    For $\Phi(t)\to \infty$ any function tending to infinity as $t\to\infty$, the set
    \[
    \tilde{\Gamma}=\left\{ \xi\in\mathbb{R}:\;\limsup_{Q\to\infty} Q \Phi(Q) \min_{q\in \mathbb{Z},0<q\le Q} \Vert q\xi\Vert \cdot \vert q \vert_{S}
     = \infty\right\}
    \]
    is a dense $G_{\delta}$ set, it has full packing dimension
    and satisfies $\tilde{\Gamma}+\tilde{\Gamma}=\tilde{\Gamma}\cdot \tilde{\Gamma}=\mathbb{R}$.
    \item[(iii)] If $S$ has singleton complement, i.e. $|S^c|=1$,
    then for any real number $\xi$ we have
     \[
    \lim_{Q\to\infty} Q \min_{q\in \mathbb{Z},0<q\le Q} \Vert q\xi\Vert \cdot \vert q \vert_{S}
     = 0.
    \]
\end{itemize}
\end{theorem}

\begin{remark}
    Presumably (iii) can be generalized, at least to the setting of $S^c$ being finite. However, the situation becomes more intricate and we do not put effort to prove it rigorously. On the other hand, it is very unclear to us if (i) holds for $|S|=2$.
\end{remark}

Note 
that the non-uniform version (i.e. the classical $S$-arithmetic Littlewood problem) of claim (ii)
is false for certain $\Phi(t)\to\infty$ as soon as $S$ has at least two elements. Indeed by Theorem 1.8 of~\cite{ven} the claim is false, so the lower limit is zero,
when $\Phi(Q)=(\log \log\log Q)^{\kappa}$ for some positive small $\kappa$.

We further remark that our method cannot provide full or even positive Hausdorff dimension in (ii) for functions $\Phi$ tending to $\infty$ at some rate of interest (where the claim is not superseded by well-known metrical results for classical non-uniform settings), 
since it then requires $\xi$ to be a Liouville number
which form a set of Hausdorff dimension zero.

As for the claims in~\S\ref{s1.2},
the proof of Theorem~\ref{t2} is considerably easier than for Theorem~\ref{T01}, we give a brief upshot:
In (i) we only require $\xi$ to have many good approximations whose denominators are powers of the prime $p$ that equals $S$. Similarly,
for (ii) we just need $\xi$ to be ``Liouville enough'', however here with denominators of well approximating convergents not divisible by primes in $S$. Claim (iii)
follows from considering two consecutive
convergents to $\xi$ and studying their linear combinations, and showing that most denominators of them are not divisible by the single prime forming $S^c$.

  \section{Some related observations on sumsets}
     Motivated by the sumset result in Theorem~\ref{T01}, we investigate the sumset to
     the counterexamples to the classical Littlewood Conjecture \eqref{eq:lw}. Recall the notation $\mathcal{A}$ from~\S~\ref{intro} for this set. 
     We show that $\mathcal{A}+\mathcal{A}$ is small in metrical sense, thereby contrasting it from the set $\Theta$ in Theorem~\ref{T01}.  Write 
     \begin{equation} \label{eq:FGH}
     Bad_k= \left\{ (\xi_1,\ldots,\xi_k)\in\mathbb{R}^k:\; \liminf_{q\to\infty} q^{1/k} \max_{1\le j\le k} \Vert q\xi_j\Vert > 0\right\}
     \end{equation}
     for the set of $k$-dimensional badly approximable real vectors.

     \begin{theorem}  \label{thmlwood}
         The set
         \[
         Bad_2+\mathcal{A}\supseteq \mathcal{A}+\mathcal{A}
         \]
         has $2$-dimensional Lebesgue measure $0$. Moreover, if $\Pi:(x,y)\to x$ is the projection to the first coordinate then the sets
         \[
         Bad_1+\Pi(\mathcal{A})\supseteq \Pi(\mathcal{A}+\mathcal{A}),\qquad
          Bad_1\cdot \Pi(\mathcal{A})\supseteq \Pi(\mathcal{A})\cdot \Pi(\mathcal{A})
         \]
         have $1$-dimensional Lebesgue measure $0$.
     \end{theorem}

      To the best of our knowledge, Theorem~\ref{thmlwood} has not been observed before.
     Note that it is not directly implied by the well-known Hausdorff dimension zero result for $\mathcal{A}$ from~\cite{ein}: Indeed, the set of Liouville numbers (or vectors) $\mathcal{L}_k$ via Proposition~\ref{erpro} provides a counterexample.
     We remark that the latter claim about $\Pi(\mathcal{A}+\mathcal{A})$ implies the first about $\mathcal{A}+\mathcal{A}$, however for the left hand supersets such a direct argument does not work as the relation between $\Pi(Bad_2)$ and $Bad_1$ is unclear, with difference set presumably large.

    \begin{proof}
     First notice that $\mathcal{A}\subseteq Bad_2 \cap (Bad_1\times Bad_1)$, so all stated inclusions are clear.
     Let us first show the second claim. Write 
     \[
     Bad_1= \cup_{m\ge 1} F_m
     \]
     where $F_m$ is the set of real numbers
     with partial quotients bounded above by $m$. Then
     \[
     Bad_1+\Pi(\mathcal{A})= \cup_{m\ge 1}
     (F_m+\Pi(\mathcal{A})).
     \]
     Now it is known that $\mathcal{A}$ and thus $\Pi(\mathcal{A})$ have Hausdorff dimension zero~\cite{ein}, moreover that
     $\dim_P(F_m)=\dim_H(F_m)<1$ for all $m$, 
     see~\cite[\S~2]{maul}. Thus by non-increase of Hausdorff dimension under Lipschitz maps and a result of Tricot~\cite{tricot}, each set $F_m+\Pi(\mathcal{A})$ has Hausdorff dimension
     \[
     \dim_H( F_m+\Pi(\mathcal{A}) )\le
     \dim_H( F_m\times \Pi(\mathcal{A}) )
     \le \dim_H(\Pi(\mathcal{A}))+ \dim_P(F_m)<1
     \]
     strictly less than one, hence zero Lebesgue measure. Analogously for the product set. By sigma-additivity of Lebesgue measure the second claim follows. The first claim is proved analogously, using a countable cover with level sets of $Bad_2$ in place of $F_m$.
    \end{proof}

    We only needed that $\mathcal{A}$ has Hausdorff dimension zero for the argument (for $Bad_2+\mathcal{A})$. Hence we can apply the analogous argument to the set of Liouville vectors $\mathcal{L}_k$, thereby we get the following result which again appears to be new even for $k=1$.

    \begin{theorem}
        For $k\ge 1$, the sumset $\mathcal{L}_k+Bad_k$
        has $k$-dimensional Lebesgue measure $0$.
    \end{theorem}

     This result may come as a surprise in view of
     \begin{equation} \label{eq:ERDE}
         \mathcal{L}_k+\mathcal{L}_k=
     \mathbb{R}^k=Bad_k+Bad_k, \qquad k\ge 1. 
     \end{equation}
      The left identity of 
    \eqref{eq:ERDE}
     follows from $\mathcal{L}_k$ being dense $G_{
     \delta}$ and Proposition~\ref{erpro}.
       The right identity in \eqref{eq:ERDE} follows from the winning property of $Bad_k$ (and consequently any set $\mathbf{x}-Bad_k$ for $\mathbf{x}\in\mathbb{R}^k$ and countable intersections of them), see~\cite{schmidt}.

    \section{Proof of Theorem~\ref{T01}: The main lemma}  \label{below}

    Theorem~\ref{T01} will employ Lemma~\ref{lemur} below.
  We introduce some notation.
    Recall $F_m$ is the set of real numbers $x$ in $(0,1)$ that when written
    as continued fraction $x=[a_1,a_2,\ldots]$ satisfy $a_i\le m$ for all partial quotients. Denote $\{x\}=x-\lfloor x\rfloor$ the fractional part of $x$. Let us further for $\tau>1$ write $a\asymp_{!\tau} b$ if $1/\tau<a/b< \tau$,
   that is $a,b$ differ by a factor less than $\tau$.

    \begin{lemma}  \label{lemur}
        There exist $\tau>1$ and $m\in\mathbb{N}$ such that
        for $I, J$ any subintervals of $(0.1,0.9)$ of positive length, there exist (arbitrarily large) positive integers $p,q,r,s$ such that
        \begin{equation}  \label{eq:C1}  \tag{C1}
        q \asymp_{!\tau} s    
        \end{equation}
        and
          \begin{equation}  \label{eq:C2}  \tag{C2}
        \frac{p}{q}\in I, \qquad \frac{r}{s}\in J
        \end{equation}
        and $ps/q$ and $qr/s$ are reduced rationals satisfying
          \begin{equation}  \label{eq:C3}  \tag{C3}
        \left\{  \frac{ps}{q} \right\} \in F_m, \qquad
        \left\{  \frac{qr}{s} \right\} \in F_m.
        \end{equation}
        We may choose $\tau$ arbitrarily close to $1$ and $m=50$.
    \end{lemma}

   The conditions \eqref{eq:C1} resp. \eqref{eq:C3} become weaker as we increase $\tau$ resp. $m$. 
   Note that the conditions \eqref{eq:C1}, \eqref{eq:C2} imply
   \[
        p\asymp q \asymp r \asymp s
        \]
   with implied constant some modification of $\tau$. The choice of $(0.1,0.9)$ is not of major significance, we should just stay away from $0$.
   The most critical condition is clearly \eqref{eq:C3}. 

   In~\S~\ref{lemproof}, we prove the lemma.
   In \S~\ref{lise} we derive Theorem~\ref{T01} from Lemma~\ref{lemur} via
   taking $\tau$ close enough to $1$ and $m=50$,
   %
%
%
%
    the smallest value we can provide  unconditionally, derived via~\cite[Remark 1.20]{bko}.
   This is far off the optimal expected value $m=2$, which would lead by an analogous argument to some optimal constant $0.0692\ldots>1/15$ in Theorem~\ref{T01} feasible by our method.
    If true, this
would imply
    strict inequality and an effective gap
    compared to the largest possible classical Littlewood constant. Indeed then
    \[
     \sup_{(\xi,\zeta)\in\mathbb{R}^2 }\;\limsup_{Q\to\infty} Q \min_{q\in \mathbb{Z},0<q\le Q} \Vert q\xi\Vert\cdot \Vert q \zeta\Vert >
     \frac{1}{15}>\frac{1}{19} \ge \sup_{(\xi,\zeta)\in\mathbb{R}^2 }\;\liminf_{Q\to\infty} Q \min_{q\in \mathbb{Z},0<q\le Q} \Vert q\xi\Vert\cdot \Vert q \zeta\Vert,
    \]
     where the most right estimate is due to Badziahin~\cite{badz}.  Clearly such a gap is a trivial consequence of Theorem~\ref{T01} if the Littlewood Conjecture \eqref{eq:lw} is true.
    An intermediate lower bound in the uniform problem larger than $0.0414>1/25$ could be derived from $m=5$, possibly more in reach for a proof as
    suggested
    by work of Frolenkov, Khan~\cite[Theorem~2.1]{frokan} following up on~\cite{bko}. However,
    it appears that at least significantly more work is required for a rigorous proof that $m=5$ suffices. Generally given we can confirm Lemma~\ref{lemur} for some $m$, our resulting bound is larger than $(m+2)^{-1}/4$, where the factor $1/4$ can certainly be improved
    without major effort.

\section{Proof of Lemma \ref{lemur} } \label{lemproof}

%
We need the following
counting lemma on the cardinality 
of product sets in finite fields
whose proof is based on Erd\H{o}s Turan inequality.

\begin{lemma}  \label{dadada}
    Let 
    $\sigma\in (0,1), d>0$ be fixed. For $N$ a large
    prime,
    let $\mathscr{M}\subseteq \{0,1,\ldots,N-1\}$ be any subset of residue classes modulo $N$ of cardinality $M=|\mathscr{M}|\ge N^{\sigma}$.
    Let $\mathcal{H}=\{L+1,\ldots,L+H\}$ be some subinterval of 
    $\{0,1,\ldots,N-1\}$ of length $H\ge dN$. 
    Then for some absolute $\eta>0$ independent of $N$, the set $\mathscr{E}$ consisting of
    $e\in \{0,1,\ldots,N-1\}$ 
    for which the set $e^{-1}\cdot \mathcal{H}=\{e^{-1}h\bmod N:h\in \mathcal{H}\}$
    has empty intersection with $\mathscr{M}$ is of cardinality
    \[
    |\mathscr{E}|=:E\ll N^{1-\eta}, \qquad \eta>0.
    \]
    
\end{lemma}

Taking the inversion $e^{-1}$ in the set
$e^{-1}\cdot \mathcal{H}$
is just to ease the proof below.

\begin{proof}
We prove the bound
\[
E\ll \frac{N^3}{H^2M}\cdot (\log N)^2\le \frac{N^{1-\sigma}}{d^2}\cdot (\log N)^2,
\]
then taking arbitrary $\eta<\sigma$ to compensate the logarithmic term 
we are done. 
We identify that $e\in \mathscr{E}$ if it is not representable as 
\[
e=h\textswab{m}^{-1}, \qquad h\in\mathcal{H},\; \textswab{m}\in\mathscr{M}.
\]
Let $e_N(x)=\exp(2\pi ix/N)$. Consider
\[
S(\lambda)= \sum_{e\in \mathscr{E}}\sum_{ \textswab{m}\in \mathscr{M}} e_N(\lambda e\textswab{m}).
\]
By a classical result of Vinogradov~\cite[Chapter VI]{vin} we have
\begin{equation} \label{eq:SS}
|S(\lambda)| \le \sqrt{NEM}.
\end{equation}
Let $\Delta$ be the normalized discrepancy of the set of fractional parts
\[
\Gamma= \left\{ \{ e\textswab{m}/N\}:\;  e\in\mathscr{E},\; \textswab{m}\in\mathscr{M} 
\right\}.
\]
By \eqref{eq:SS} applied to $\lambda=1,2,\ldots,N-1$ and Erd\H{o}s Turan inequality, we get
\begin{equation} \label{eq:Del}
\Delta\ll \frac{ \sqrt{NEM} \log N}{EM}= N^{1/2} E^{-1/2} M^{-1/2} \log N.
\end{equation}
In particular the interval $[(L+1)/N,(L+H)/N]$ 
contains 
\begin{equation} \label{eq:keq}
K=\frac{H}{N} EM + O(\Delta EM)
\end{equation}
elements of the set $\Gamma$.
On the other hand $K=0$ by definition of $\mathscr{E}$. Thus from \eqref{eq:Del}, \eqref{eq:keq} we get
\[
\frac{H}{N} \ll \Delta \ll N^{1/2} E^{-1/2} M^{-1/2} \log N.
\]
Rearrangements yield indeed
\[
E\ll \frac{N^3}{H^2M} (\log N)^2.
\]
%
%
%
%
\end{proof}

\begin{remark}
If $\mathcal{H}=\mathcal{H}_0:=\{1,2,\ldots,H\}$ meaning $L=0$, we could have directly applied~\cite[Theorem 1.1]{garspa}, whose conditions are easily checked to apply.
For general intervals we needed another strategy. 
\end{remark}


\begin{remark}  \label{remark5}
    The claim would be in general 
    false for $\mathcal{H}$ an arbitrary set
    of cardinality $\ge dN$, as for example 
    if $N$ is prime
    we could take $\mathcal{M},\mathcal{H}$ subsets
    of a small subgroup of the multiplicative group modulo $N$. Then $\mathscr{E}$ is also contained in this subgroup.
    A similar argument shows that primality of $N$ is necessary for the conclusion.
\end{remark}

Call $u$ a Zaremba $m$ numerator for $q$ if
$u/q\in F_m$. Call $\mathscr{M}_q\subseteq \{1,2,\ldots,q-1\}$ this set of Zaremba numerators for a given $q$.

A direct consequence of Bourgain, Kontorovich \cite[Theorem 1.8]{bko} together with the Prime Number Theorem can be stated as follows.

\begin{theorem}[Bourgain, Kontorovich] \label{bkthm}
Fix $\gamma\in (0,1)$.
For $m$ large enough and $T$ any large parameter, the following holds: 
A density one set of the primes $q\in [\gamma T,T]$ has the property that
with $\sigma=2\delta-1.001>0$,
where $\delta=\dim_H(F_m)$,
there are $|\mathscr{M}_q|\gg T^{\sigma}$ many Zaremba $m$ numerators for $q$. In fact the 
set of counterexamples $q$ has cardinality
$\ll T^{1-c/\log\log T}$ for some $c>0$.
\end{theorem}

Fix for now $\gamma\in (0,1)$.
For large $T$,
take the set of primes $\mathcal{P}=\mathcal{P}(T)\subseteq [\gamma T,T]\cap \mathbb{Z}$ of cardinality $|\mathcal{P}(T)|=(1-\gamma-o(1))T/\log T$ from Theorem~\ref{bkthm}.
For each such prime $q\in \mathcal{P}$, we may apply Lemma~\ref{dadada} to
\[
N=q,\quad \mathscr{M}=\mathscr{M}_q,
\quad \mathcal{H}=Iq\cap \mathbb{Z},\quad
d=|I|,
\quad \sigma=2\delta-1.001>0.
\]
Since for each $q\in\mathcal{P}$ the exceptional set $\mathscr{E}_q$ in Lemma~\ref{dadada} has cardinality 
\[
|\mathscr{E}_q|\le N^{1-\eta}\le T^{1-\eta}=o(T/\log T)
\]
negligible compared to the total number of primes in $[\gamma T,T]$,
for each $q\in \mathcal{P}$
there are $\ge (1-\gamma-o(1))T/\log T$ many primes $s=s(q)\in [\gamma T,T]$ in the complementary set $\mathscr{E}_q^c$, for each of which there exists $p=p(q,s)$ so that 
$\{ ps/q\}\in F_m$ for the absolute $m$ above (we may reduce $s\bmod q$ wherever needed, it is easily checked this does not affect any counting below). 
By removing from $\mathscr{E}_q^c\cap [\gamma N,N]$
another negligibly small set 
of cardinality $\le N^{1-c/\log\log N}=o(N/\log N)$ hence also $o(T/\log T)$
if necessary, we may assume that all these primes $s\in \mathscr{E}_q^c\cap [\gamma N,N]$ satisfy Theorem~\ref{bkthm} again.
By pigeon hole principle there is some fixed $s\in \cup_q \mathscr{E}_q^c\cap [\gamma T,T]$, call it $\tilde{s}$, so that the above applies
for a large subset of values $q$ in $\mathcal{P}$ uniformly, call this set $\mathcal{Z}_1\subseteq \mathcal{P}$,
of cardinality $|\mathcal{Z}_1|\ge (1-\gamma-o(1))T/\log T$.
(Indeed in total we have $((1-\gamma)^2-o(1))T^2/(\log T)^2$ many pairs $(q,s)$
partitioning into $(1-\gamma-o(1))T/\log T$ many second coordinates $s$, so some $s$ must have at least $(1-\gamma-o(1))T/\log T$ many hits by first coordinates $q$.)  
As noticed above, by construction, given $q\in \mathcal{Z}_1$, we find some integer $p=p(q,\tilde{s})$ with $\{ p\tilde{s}/q\}\in F_m$. On the other hand, as $\tilde{s}$ satisfies Bourgain-Kontorovich property from Theorem~\ref{bkthm},
again by Lemma~\ref{dadada} applied to
\[
N=\tilde{s},\quad \mathscr{M}=\mathscr{M}_{\tilde{s}},\quad \mathcal{H}=J\tilde{s}\cap \mathbb{Z},\quad d=|J|,\quad \sigma=2\delta-1.001>0
\]
only few integers $\ell\in [1,T]$ can have have the property that $\{\ell\cdot (J\tilde{s}\cap\mathbb{Z})\bmod \tilde{s}\}$ has empty intersection with $\mathscr{M}_{\tilde{s}}$.
Call this set $\mathcal{Z}_2=\mathscr{E}_{\tilde{s}}$ which has cardinality $|\mathcal{Z}_2|=T^{1-\eta}=o(T/\log T)$.
But this means there remain many $q\in\mathcal{Z}_1\setminus \mathcal{Z}_2$, more precisely a set of cardinality $|\mathcal{Z}_1\setminus \mathcal{Z}_2|\ge (1-\gamma-o(1))N/\log N$, for which 
we find $r$ such that $rq/\tilde{s}\in F_m$ as well.
So to sum up, we can take pairs (of primes) $q,s$ where both conditions in \eqref{eq:C3} hold for some $p,r$. Moreover condition
\eqref{eq:C1} holds with $\tau=\gamma^{-1}$ since $q,s\in [\gamma T,T]$, which as
$\gamma\to 1$ indeed
can be chosen arbitrarily close to $1$.
Clearly \eqref{eq:C2} holds by construction as well. 

\begin{remark}
    We once again stress that the main restriction for the bounds in our results comes from Theorem~\ref{bkthm}.
    Currently $\delta>307/312$ is required
    leading to $m=50$, see~\cite[Remark~1.20]{bko}.
    Ideally $\delta>0.5$ and hence $m=2$ suffices, since $\dim_H(F_2)>0.53$ is well-known (see~\cite[Remark~1.17]{bko} for references), which leads to the conjectural
    lower bound $>1/15$ discussed in \S~\ref{intro}. We also mention
    that small improvements in the final bound
    can possibly be made by restricting the 
    digit patterns on $\{1,2,\ldots,m\}$
    in order to improve the factor $(m+2)^{-1}$ derived from Proposition~\ref{p2}, e.g. excluding the string $1m1$.
    This would require a variant of Theorem~\ref{bkthm} encoding these restrictions, moreover the Hausdorff dimension of the resulting set needs 
    to be large enough (probably $>307/312$ again) for the conclusion.
    Clearly the best we can hope for is replacement of $(m+2)^{-1}$ by $m^{-1}$,
    probably $(m+1)^{-1}$ being more realistic.
\end{remark}


    \section{Deduction of Theorem~\ref{T01}
    from Lemma~\ref{lemur} } \label{lise}

    We first state an easy but crucial observation.

    \begin{proposition} \label{pro}
    Let $c,d$ positive real numbers with $cd<1$.
        Let $\xi$ be a real number and assume for some integers $p$ and $q>0$ we have
        \[
        |q\xi-p| < cq^{-2}.
        \]
        Then for any integers $u,v$ so that $u/v\ne p/q$ and 
        $0<v<Q=dq^2$, we have
        \[
        |v\xi-u| \ge (1-cd)\sqrt{d} Q^{-1/2}.
        \]
    \end{proposition}

    \begin{proof}
        Consider the modulus of the 
        determinant of the matrix 
        \[
        M:= \begin{pmatrix}
            p & q \\
            u & v
        \end{pmatrix}.
        \]
        It is at least $|\det M|\ge 1$ 
        by $p/q\ne u/v$, on the other hand subtracting 
        $\xi$ times the second column from the first column 
        and expanding, we get
        it is at most
        \[
        |\det M|\le q|v\xi-u| + Q cq^{-2} \le q|v\xi-u| +cd.
        \]
        Combining these bounds we get
        \[
        |v\xi-u|\ge (1-cd)q^{-1}= (1-cd)\sqrt{d} Q^{-1/2} .
        \]
    \end{proof}

    The following is well-known~\cite{Perron}.
 
    \begin{proposition} \label{p2}
        We have $|q_kx-p_k|\ge (a_{k+1}+2)^{-1}q_k^{-1}$ for $p_k/q_k=[a_1,\ldots,a_k]$ any convergent to some real $x\in (0,1)$.
    \end{proposition}

 Now we prove Theorem~\ref{T01}.
    We will construct $\xi, \zeta$ by
    \begin{equation} \label{eq:inter}
    \xi=\bigcap I_j, \qquad \zeta=\bigcap J_j
    \end{equation}
    with two sequences of fast shrinking nested intervals to which in each step we apply Conjecture~\ref{lemur}.
    Let $I_0, J_0\subseteq (0.1,0.9)$ arbitrary of positive length be given.
    Choose suitable rationals $p_0/q_0\in I_0$ 
    and $r_0/s_0\in J_0$ written
    in lowest terms satisfying the conditions 
    of Lemma~\ref{lemur} for given $\tau,m$ with $I=I_0$, $J=J_0$, thus in particular
    so that $q_0\asymp_{!\tau} s_0$ with some fixed $\tau>1$. 
    Fix for now arbitrary constants $d, \alpha,\beta$ satisfying
    \begin{equation} \label{eq:cstr}
    d\in(0,1], \qquad 0<\alpha<\beta< \min\{ \frac{\tau}{ d^2(m+2) } , 1/d\}.
    \end{equation}
    Now choose
    that next intervals for $\xi, \zeta$ so that
    \begin{equation} \label{eq:1}
    \alpha q_0^{-3}\le |\xi-p_0/q_0|\le \beta q_0^{-3}, \qquad 
    \alpha s_0^{-3}\le|\zeta-r_0/s_0|\le \beta s_0^{-3},
    \end{equation}
    or equivalently
    \[
    \alpha q_0^{-2}\le |q_0\xi-p_0|\le \beta q_0^{-2}, \qquad 
    \alpha s_0^{-2}\le |s_0\zeta-r_0|\le \beta s_0^{-2}.
    \]
    Considering only the intervals to the right of $p_0/q_0$ resp. $r_0/s_0$, this means we let
    \begin{equation} \label{eq:drueber}
    I_1=[p_0/q_0+\alpha q_0^{-3},p_0/q_0+\beta q_0^{-3}], \quad
    J_1= [r_0/s_0+\alpha s_0^{-3},r_0/s_0+\beta s_0^{-3}].
    \end{equation}
    We may assume that $I_1\subseteq I_0, J_1\subseteq J_0$, otherwise (if $p_0/q_0$ resp. $r_0/s_0$ are very close to the boundaries of $I_0$ resp. $J_0$) we argue below
    for slightly smaller subintervals of $I_0,J_0$.
    Let 
    \[
    Q=Q_0= \tau^{-1}d q_0s_0 \in (\tau^{-2}dq_0^2, dq_0^2)\cap (\tau^{-2}ds_0^2, ds_0^2),
    \]
    where we used condition \eqref{eq:C1}.
    From Proposition~\ref{pro}
    with $c=\beta$ and the present $d$
    we get that
    if $(u,v)\in\mathbb{Z}^2$ with $1\le u\le Q$ is not a multiple of $(p_0,q_0)$ then
    \[
    |u \xi- v| \ge (1-\beta d)\sqrt{d} Q^{-1/2}
    \]
    similarly if $(u,w)\in\mathbb{Z}^2$ is not multiple of $(r_0,s_0)$ then
    \[
    |u \zeta- w| \ge (1-\beta d)\sqrt{d} Q^{-1/2}.
    \]
     Hence if both applies then
     \begin{equation}  \label{eq:D1}
     Q |u \xi- v|\cdot |u \zeta- w| \ge (1-\beta d)^2d>0.
     \end{equation}
     Thus we are only left with
     the excluded two types of multiples.

     So let first $u=q_0y$ and $v=p_0y$ for some integer $1\le y\le Q/q_0$. Then
     \[
     |u \xi- v|= y |q_0 \xi- p_0|\ge \alpha y q_0^{-2}\ge 
     \frac{\alpha}{\tau^2}d\cdot yQ^{-1}.
     \]
     Hence if it is true that for some absolute constant
     $C>0$
     \begin{equation} \label{eq:c1}
     |u \zeta- w|= |y\cdot q_0\zeta-w|\ge   
     Cy^{-1}
     \end{equation}
     for all integers $w$, then again we can conclude
     \begin{equation} \label{eq:D2}
     Q |u \xi- v|\cdot |u \zeta- w| \ge \frac{C\alpha}{\tau^2}d>0.
     \end{equation}
     The other case $u=s_0y$ and $w=r_0y$ works analogously,
     assuming for $1\le y\le Q/s_0$ we have
     \begin{equation} \label{eq:c2}
     |u \xi- z|= |y\cdot s_0\xi-z|\ge  
     Cy^{-1}
     \end{equation}
     for all integers $z$, we again conclude the analogue of \eqref{eq:D2}.
     We show that the conditions of Conjecture \ref{lemur} imply these estimates \eqref{eq:c1}, \eqref{eq:c2}
     for some explicitly computable $C$.
     
     By condition \eqref{eq:C1} of Lemma~\ref{lemur}
     for our $\tau$ chosen above
     we see $Q/q_0=ds_0/\tau<dq_0\le q_0$ and $Q/s_0=dq_0/\tau<ds_0\le s_0$.
     Hence by condition \eqref{eq:C3} 
     of Lemma \ref{lemur} and Proposition \ref{p2} we have that
     \[
     \Vert ys_0(p_0/q_0)\Vert\ge (m+2)^{-1}y^{-1}, \qquad 
     \Vert yq_0(r_0/s_0)\Vert\ge (m+2)^{-1}y^{-1}
     \]
     for all integers $1\le y\le Q/s_0$ resp.
     $1\le y\le Q/q_0$. However recall \eqref{eq:1} implies
     \[
     \vert \xi-p_0/q_0\vert\le \beta q_0^{-3}, \qquad
     \vert \zeta-r_0/s_0\vert\le \beta s_0^{-3},
     \]
     so by condition \eqref{eq:C1} of Lemma \ref{lemur} 
     and $y\le Q/s_0=\tau^{-1}dq_0$
     we can
     estimate 
     \[
     ys_0\vert \xi-p_0/q_0\vert\le 
     (\tau^{-1}dq_0)s_0\cdot \beta q_0^{-3}=\tau^{-1}d\beta(s_0/q_0)q_0^{-1}<
     \tau^{-1}d\beta \tau q_0^{-1}\le
     \beta\tau^{-1}d^2 y^{-1}
     \]
     and further that
     \[
     \Vert ys_0\xi\Vert
     =\Vert ys_0[p_0/q_0+(\xi-p_0/q_0)]\Vert
     \ge
     \Vert ys_0(p_0/q_0)\Vert-ys_0\vert \xi-p_0/q_0\vert
     \ge
     [(m+2)^{-1}-\beta\tau^{-1}d^2]y^{-1}.
    \]
    Very similarly we derive
    \[
     \Vert yq_0\zeta\Vert\ge [(m+2)^{-1}-\beta\tau^{-1}d^2]y^{-1}
     \]
     for all $y\le Q/q_0<s_0$ as above. Thus we 
     conclude that \eqref{eq:c1} and \eqref{eq:c2} hold for
     \[
     C=(m+2)^{-1}-\beta\tau^{-1}d^2>0.
     \]
     %
      %
     Combining the two cases \eqref{eq:D1}, \eqref{eq:D2},
     for $Q=Q_0$ we get
     \[
     Q \min_{q\in \mathbb{Z},0<q\le Q} \Vert q\xi\Vert\cdot \Vert q \zeta\Vert 
     \ge \min\left\{ (1-\beta d)^2d\;,\; \alpha d\cdot  \frac{  \frac{1}{m+2}-\beta \tau^{-1}d^2}{\tau^2}\right\}
     >0.
     \]
      We may take $\tau$ arbitrarily close to $1$ by Lemma~\ref{lemur}, henceforth by continuity we may assume $\tau=1$.
      Moreover we can take $\alpha$ arbitrarily close to $\beta$, so similarly let us assume $\alpha=\beta$.
      Accordingly we seek to maximize
      \[
      \min\left\{ (1-\beta d)^2d\;,\; \beta d\cdot  \left( \frac{1}{m+2}-\beta d^2\right) \right\}
      \]
      under the constraints \eqref{eq:cstr}
      with $\tau=1$.
     Taking $m=50$ and the equilibrium point
     \begin{equation}  \label{eq:jety}
     d=0.01465, \qquad \beta= \frac{2d+\frac{1}{52}-\sqrt{ \frac{d}{13}+\frac{1}{52^2}-4d^2 }}{4d^2}=27.1023\ldots
     \end{equation}
      we readily check \eqref{eq:cstr} holds and we get a bound 
      \[
     Q \min_{q\in \mathbb{Z},0<q\le Q} \Vert q\xi\Vert\cdot \Vert q \zeta\Vert 
     >0.005326.
     \]

     This concludes the first step. 
     
     Now notice that
     via \eqref{eq:1} 
     we have freedom in the choices
     of $\xi, \zeta$ in the non-trivial intervals $I_1\subseteq I_0, J_1\subseteq J_0$ defined in \eqref{eq:drueber}. 
     Now we repeat step 1 starting from suitable fractions
     $p_1/q_1\in I_1$ and $r_1/s_1\in J_1$ satisfying the hypotheses of Lemma \ref{lemur} for the new intervals $I=I_1, J=J_1$, where the constant $C$ from \eqref{eq:c1}, \eqref{eq:c2} is the same. 
     Again 
     this gives us 
     new intervals $I_2\subseteq I_1, J_2\subseteq J_1$. We iterate the process
     ad infinitum, and finally define $\xi, \zeta$ by
     \eqref{eq:inter} which is clearly well-defined as the intervals are compact and nested. For these we get a sequence
     of values $Q_k\to\infty$ where
     \[
     \liminf_{k\to\infty} Q_k \min_{q\in \mathbb{Z},0<q\le Q_k} \Vert q\xi\Vert\cdot \Vert q \zeta\Vert>0.005326,
     \]
     thus in particular 
     \[
     \limsup_{Q\to\infty} Q \min_{q\in \mathbb{Z},0<q\le Q} \Vert q\xi\Vert\cdot \Vert q \zeta\Vert 
     >0.005326
     >0
     \]
     so $\xi,\zeta$ are not satisfying ULC. 

     Since we can start with $I_0\ni \xi, J_0\ni \zeta$ arbitrary intervals in our construction, we get that $\Theta$ is dense. 
     To see $\Theta$ is also $G_{\delta}$, note that by continuity for any $Q\ge 1$ the set 
     \[
     \Sigma_Q:= \{ (\xi,\zeta)\in\mathbb{R}^2: \; 
      Q \min_{q\in \mathbb{Z},0<q\le Q} \Vert q\xi\Vert\cdot \Vert q \zeta\Vert \le 1/188 \}
     \]
     is closed. Hence the sets $\Sigma^k:=\cap_{k\ge Q} \Sigma_Q$ where this holds for all large enough integers $Q$ are closed as well. Thus its complement $(\cup_{k\ge 1} \Sigma^k)^c=\cap_{k} (\Sigma^k)^c=\Theta$ is indeed a $G_{\delta}$ set. 
     Finally, as already pointed out in~\S~\ref{intro}, the claim on the packing dimension follows now from Lemma~\ref{icha}, the claim $\Theta+\Theta=\mathbb{R}^2$ from 
     Proposition~\ref{erpro}.

    \begin{remark}
        The values in \eqref{eq:jety} are approximations to some complicated optimal algebraic numbers, as pointed out in Remark~\ref{remark1}.
    \end{remark}

\section{Proof of Theorem~\ref{2t} }

\subsection{Preparatory results}
We will frequently use the following easy observation in the spirit of Proposition~\ref{pro}.

\begin{proposition} \label{dieprop}
    If $p/q$ is a convergent of $\xi$ and $(r,s)$ is not a multiple of $(p,q)$ and $|s|<|q\xi-p|^{-1}/3$, then
    we have
    \[
    |s\xi-r|\ge \frac{1}{3q}.
    \]
\end{proposition}

\begin{proof}
    The determinant of the matrix 
      \[
        M:= \begin{pmatrix}
            p & q \\
            r & s
        \end{pmatrix}
        \]
    has modulus $|\det M|\ge 1$ by linear independence. On the other hand if we assume the claim is false then subtracting $\xi$ times the second from the first column yields
    it is $|\det M|\le s|q\xi-p|+q|r\xi-s|\le 1/3+1/3=2/3<1$. So we get a contradiction.
\end{proof}

The next lemma is essentially known, we state a short proof for convenience.

\begin{lemma}  \label{gdelta}
    Let $g: \mathbb{N}\to (0,\infty)$ and
     $h: \mathbb{N}\to [0,\infty)$
    be arbitrary functions tending to $0$ and with $h(q)<g(q)$ for all $q\in\mathbb{N}$. Let $X\subseteq \mathbb{N}$ unbounded. Then the set
    \[
    Y:= \left\{ \xi\in\mathbb{R}:\; h(q)<|\xi-p/q|< g(q)\quad for\; infinitely\; many\; pairs\; p\in \mathbb{Z}, \; q\in X  \right\}
    \]
    is a dense $G_{\delta}$ set.
\end{lemma}

The special case $g(q)=e^{-q}, h\equiv 0$
and $X=\mathbb{N}$ implies the well-known result that the set $\mathscr{L}_1$ of Liouville numbers 
is dense $G_{\delta}$, already noticed in \S~\ref{intro}.
It can be readily generalized to simultaneous approximation.

\begin{proof}
      To see the density, given any non-empty open interval
      $I_0$, choose $q=q_0\in X$ large enough so that $I_1:=(p/q-g(q),p/q+g(q))\subseteq I_0$, actually $I_1$ should be in the interior in $I_0$. 
      Let $J_1=(p/q+h(q),p/q+g(q))\subseteq I_1$.
      Then repeat the process with $I_2:=J_1$ to generate $I_3\subseteq J_1$ and then $J_2\subseteq I_3\subseteq J_1$,
      and so on. The unique number $\xi=\cap J_i$ 
      lies in $I_0\cap Y$ and $I_0$ was arbitrary.
      For the $G_{\delta}$ property,
      by continuity,
      for any integer $Q\ge 1$
      the set 
      \[
      K_q:= \{ \xi: |p/q-\xi|\notin (h(q),g(q)),\; \forall p\in\mathbb{Z}\}
      \]
      is closed. Hence $L_{Q_0}:=\cap_{Q\ge Q_0, Q\in S} K_Q$
      is closed as well for all $Q_0$. Thus
      $\cup_{Q\ge 1} L_{Q}$ is an $F_{\sigma}$ set.
      However 
      its complement is exactly the set $Y$ of the lemma, consequently a $G_{\delta}$ set. 
\end{proof}

\subsection{Proof of Theorem~\ref{2t}}

    For the moment let $\Psi$ be arbitrary.
Let $\boldsymbol{\zeta}=(\zeta_1,\ldots,\zeta_{m-1})\in \mathcal{A}_{m-1}(\Psi)$ be arbitrary. Then
    \begin{equation}  \label{eq:schreib}
    \prod_{i=1}^{m-1} |s\zeta_i-u_i|\ge C\cdot \Psi(s)/s
    \end{equation}
    for some $C=C(\boldsymbol{\zeta})>0$ and all integers $u_i,s\ne 0$.

We show that a dense $G_{\delta}$ set of Liouville numbers $\xi$, independent of $\boldsymbol{\zeta}$, induces a positive limsup
    for $(\xi,\zeta_1,\ldots,\zeta_{m-1})$.
    We construct suitable $\xi$ as an infinite intersection of nested compact intervals $I_j$.
    Consider $I_0\subseteq (0,1)$ is defined.
   Let $p/q=p_0/q_0\in I_0$ in lowest terms
    be arbitrary and take $Q=Q_0$ large enough that $\Phi(Q)>q$. We may pick $Q$ such that additionally
    \begin{equation}  \label{eq:truhe}
        \Psi(Q)\le K\cdot \min_{0<t\le Q} \Psi(t)
    \end{equation}
    for fixed $K>0$.
    To see this, first observe that
    if $\liminf_{t\to\infty} \Psi(t)=0$
    then we can take $K=1$ for some suitable values of $Q$, since the domain of $\Phi$ is discrete (thus compact).
    Assume otherwise $\liminf_{t\to\infty} \Psi(t)>0$. Note that conversely we may assume $\liminf_{t\to\infty} \Psi(t)<\infty$, as otherwise by Dirichlet's Theorem 
    $\mathcal{A}_{m-1}(\Psi)=\emptyset$ and the claim of the theorem is trivial. However, this means that
    \eqref{eq:truhe} holds for some absolute $K$ and any sequence of values $Q$ that realizes the liminf.
    
    We choose $\xi\in I_1:=[p/q+(1/6)q^{-1}Q^{-1},p/q+(1/3)q^{-1}Q^{-1}]\subseteq I_0$ 
    so that $p/q$ is
    a good approximation to $\xi$, more precisely
    \[
    1/6 \le |q\xi-p|Q \le 1/3.
    \]
    We study products $(s\xi-r)\prod (s\zeta_i-u_i)$ for $0<s\le Q$ and consider two cases.
    
    Case 1: $(r,s)$ is multiple of $(p,q)$.
    Write $(r,s)=N(p,q)$ with $N\le Q/q$ so that $s\le Q$. Then
    \[
    |s\xi-r| =N\cdot |q\xi-p| \ge (1/6) N Q^{-1}
    \]
    On the other hand, since $\zeta\in\mathcal{A}_{m-1}(\Psi)$, for any integers $s$ and $u_i$ by \eqref{eq:schreib}, \eqref{eq:truhe}
    \[
    \prod_{i=1}^{m-1} |s\zeta_i -u_i| \ge C\Psi(s)s^{-1}=C\Psi(s)N^{-1}q^{-1}\ge CK^{-1}\Psi(Q)N^{-1}q^{-1}.
    \]
    So the product satisfies
    \[
    |s\xi-r| \cdot \prod_{i=1}^{m-1} |s\zeta_i -u_i| \ge \frac{C}{6K}\Psi(Q)Q^{-1}q^{-1}> \frac{C}{6K}Q^{-1}\Psi(Q)/\Phi(Q)
    \]
    and so
     \begin{equation} \label{eq:rest2}
    \frac{Q \Phi(Q)}{\Psi(Q)} \min_{q\in \mathbb{Z},0<q\le Q} \Vert q\xi\Vert\cdot \Vert q \zeta\Vert  \ge \frac{C}{6K}> 0.
    \end{equation}

     Case 2: $(r,s)$ is linearly independent  with $(p,q)$. Then since $0<s\le Q\le |q\xi-p|^{-1}/3$ we may apply Proposition~\ref{dieprop} which yields
    \[
    |s\xi-r|\ge q^{-1}/3> \Phi(Q)^{-1}/3.
    \]
    On the other hand again for any integers $u_i$ in view of \eqref{eq:schreib}, \eqref{eq:truhe} we have
    \[
    \prod_{i=1}^{m-1} |s\zeta_i -u_i| \ge C\Psi(s)s^{-1}\ge \frac{C}{K}\Psi(Q) Q^{-1}.
    \]
   Taking the product, 
    again \eqref{eq:rest2} follows for $Q=Q_0$.
    This concludes the first step. Now
    we take a new arbitrary $p/q=p_1/q_1\in I_1$ 
    and again take $Q=Q_1$ so that
    $\Phi(Q)>q_1$ and
    define $I_2=[p/q+(1/6)q^{-1}Q^{-1},p/q+(1/3)q^{-1}Q^{-1}]\subseteq I_1$
    and proceed as above. Continuing this process, this induces a sequence $Q=Q_k\to\infty$ for which \eqref{eq:rest2} holds for $\xi=\cap_{j\ge 1} I_j$.  

    To see $\mathcal{G}$ is a dense $G_{\delta}$ set, note that any number satisfying $|p/q-\xi|< (1/3)/\Phi^{-1}(q)$ for infinitely many integer pairs $p,q$ lies in $\mathcal{G}$ (upon choosing suitable $Q$), where $\Phi^{-1}(x)=\min\{t: \Phi(y)\ge x,\; \forall y\ge t\}$ denotes essentially the inverse function. We conclude by Lemma~\ref{gdelta} that indeed
     \[
     \mathcal{A}_{m-1}(\Psi)\times \mathcal{G}\subseteq
    \mathcal{B}_m(\Psi/\Phi).
    \]

\section{Proof of Theorem~\ref{t2} }

Proof of (i): Write $S=\{p\}$. Let for now fixed $0<\beta<1$ be fixed.
Let $\xi$ be any real number so that for some integer sequences $a_k\to\infty$ and $u_k$ with $p\nmid u_k$, that may depend on $\xi$, we have 
\begin{equation} \label{eq:Neu}
\vert p^{a_k}\xi-u_k\vert< \beta p^{-a_k}.
\end{equation}
Let $Q_k=p^{a_{k}}-1$. 
Given an integer $0<s\le Q_k$, write 
$s=mp^v$ with $p\nmid m$ and $m\le Q_k/p^v\le p^{a_k-v}$, so that
 \[
 |s|_p= p^{-v}.
\]
On the other hand, since $u_k/p^{a_k}-r/s\ne 0$ has denominator dividing $s\cdot p^{a_k-v}$
\[
|u_k/p^{a_k}-r/s|\ge (sp^{a_k-v})^{-1} 
=s^{-1}\cdot p^{v-a_k},
\]
hence
\[
|\xi-r/s|\ge |u_k/p^{a_k}-r/s|-|\xi-u_k/p^{a_k}|
\ge s^{-1}\cdot p^{v-a_k}
-\beta p^{-2a_k}.
\]
Thus for large $k$ we get
\begin{align*}
|s|_p\cdot |s\xi-r| &\ge p^{-v}(p^{v-a_k}-\beta sp^{-2a_k})=p^{-a_k}-\beta sp^{-2a_k-v}
\\
&\ge p^{-a_k}-\beta p^{-a_k}= (1-\beta)p^{-a_k}\ge (1-2\beta) Q_k^{-1}.
\end{align*}
Thus it follows that
\[
\limsup_{Q\to\infty} Q\min_{0<s\le Q} |s|_p\cdot |s\xi-r|\ge 1-2\beta.
\]
By choosing instead a sequence $\beta_k\to 0$
the right hand side can be made exactly $1$.
The reverse inequality for the limsup in $\Gamma$ is clear by Dirichlet's Theorem.

The set of $\xi$ with the property \eqref{eq:Neu} are dense $G_{\delta}$ by Lemma~\ref{gdelta}, moreover it has Hausdorff dimension $1/2$ by Borosh and Fraenkel~\cite{bofr}. Hence the same applies to $\Gamma$.
    Thus $\Gamma$ has full packing dimension by Lemma~\ref{icha}, and $\Gamma+\Gamma=\Gamma\cdot \Gamma=\mathbb{R}$
    follows from Proposition~\ref{erpro}. 

Proof of (ii):
Assume without loss of generality $2\notin S$, for any other excluded prime the below proof works analogously.
    Assume $\xi$ is any number such that
    for integer sequences $a_k\to\infty$ and
    $p_k$
    \begin{equation} \label{eq:phioben}
    \frac{1}{6\Phi^{-1}(2^{a_k})}<|2^{a_k}\xi-p_k|<  \frac{1}{3\Phi^{-1}(2^{a_k})},
    \end{equation}
    where $\Phi^{-1}(x)=\min\{t\in\mathbb{N}: \Phi(y)\ge x,\; \forall y\ge t\}$ for $x\in\mathbb{R}_{>1}$ denotes essentially the inverse function.
    Let
    \[
    Q_k=\Phi^{-1}(2^{a_k}).
    \]
    Note that then
    \[
    \Phi(Q_k)\ge 2^{a_k} ,\qquad  Q_k^{-1}/3>|2^{a_k}\xi-p_k|> Q_k^{-1}/6.
    \]
    Write $q_k=2^{a_k}$. We distinguish two cases:\\
    
    Case 1: $(r,s)$ is a multiple of $(p_k,q_k)$. Then
    $r=yp_k, s=yq_k$ for an integer $y\ge 1$ and
    since $|q_k|_{S}=1$ by construction we infer
    \begin{align*}
        Q_k\Phi(Q_k)\cdot |s\xi-r|\cdot |s|_{S}&\ge Q_k\Phi(Q_k)\cdot |yq_k-yp_k|\cdot |yq_k|_{S}=
    Q_k\Phi(Q_k)\cdot y |q_k\xi-p_k| \cdot |y|_{S}\\
    &\ge Q_k\Phi(Q_k)\cdot |q_k\xi-p_k|\ge Q_k\Phi(Q_k)\cdot Q_k^{-1}/6= \Phi(Q_k)/6\gg 1. 
    \end{align*}

    Case 2: $(r,s)$ is not a multiple of $(p_k,q_k)$. Then since
    \[
    s\le Q_k<|q_k\xi-p_k|^{-1}/3
    \]
    we can apply Proposition~\ref{dieprop} and get
    \[
    Q_k\Phi(Q_k)|s\xi-r|\cdot |s|_{S}\ge Q_k\Phi(Q_k)\cdot (q_k^{-1}/3)\cdot Q_k^{-1}= \Phi(Q_k)q_k^{-1}/3\ge \frac{1}{3}.
    \]
    Hence we can reach a positive upper limit.
    Starting with arbitrary $\Phi$ and applying above argument to $\tilde{\Phi}=\sqrt{\Phi}$ (which tends to infinity as well but slower), we see that we can replace
    the right hand side by a sequence tending to infinity
    as $k\to\infty$.

    The remaining argument is as in (i), noting that the set satisfying \eqref{eq:phioben} again is dense $G_{\delta}$ by Lemma~\ref{gdelta}.


\begin{remark}
We may directly use Tricot's inequality from the proof of Lemma~\ref{icha} to verify full packing dimension in (i), (ii) without the intermediate 
step of showing residuality. To conclude as in the lemma, it suffices to
construct for given $x\in\mathbb{R}$ 
some $\xi$ as in Theorem~\ref{t2}, (i) or (ii),
and some Liouville number $\ell\in \mathscr{L}_1$ so that $\ell+\xi=x$.
In view of the considerations above (the $p$-adic/binary digits of $\xi$ can be freely chosen 
in long ranges),
by an easy adaption
of Erd\H{o}s' constructive 
proof of $\mathscr{L}_1+\mathscr{L}_1=\mathbb{R}$ in~\cite{erd},
suitable $p$-adic/binary 
expansions of $\ell, \xi$ can be readily
derived from the $p$-adic/binary expansion of $x$. 
Similarly the claimed sum and product set properties can be proved directly.
\end{remark}

Proof of (iii): Let $p=S^c$. For $\xi\in\mathbb{Q}$ the claim is obvious, so assume $\xi$ is irrational. Let $(r_k/s_k)_{k\ge 1}$ be the sequence of convergents to $\xi$. For given large $Q$, let $k$ be the largest integer with $s_k\le Q$. 
If $p\nmid s_k$, then $|s_k|_{S}= s_k^{-1}$
and hence by Dirichlet's Theorem
\[
\Vert s_k\xi\Vert\cdot |s_k|_{S} \le s_k^{-1}Q^{-1}
\]
we obtain
\[
Q \min_{q\in \mathbb{Z},0<q\le Q} \Vert q\xi\Vert \cdot \vert q \vert_{S}\le Q \Vert s_k\xi\Vert\cdot |s_k|_{S} \le s_{k}^{-1}
\]
and since $s_k\to\infty$ as $Q\to\infty$ we are done.
So we may assume $p| s_k$. Note that then $p\nmid s_{k-1}$ and more generally $p\nmid s_k+ts_{k-1}$ for any $t\ne 0$.

Let $t$ be the largest non-zero integer so that $s_k+ ts_{k-1}\le Q$. Note $t$ may be negative.
If $t\ge 1$, then let $q=s_k+ ts_{k-1}$.
It follows easily from continued fraction theory and $Q<s_{k+1}$ that
\[
\Vert q\xi\Vert \ll s_k^{-1}.
\]
Moreover 
\[
|q|_{S}=q^{-1}\ll Q^{-1}
\]
since $Q<s_k+(t+1)s_{k-1}<2(s_k+ts_{k-1})=2q$.
Again we conclude via $s_k\to\infty$.

Finally suppose $t<0$. Then $Q<s_k+s_{k-1}<2s_k$. Hence
\[
\Vert s_{k-1}\xi\Vert\le s_{k}^{-1}\le 2Q^{-1}.
\]
Thus again by
\[
|s_{k-1}|_S= s_{k-1}^{-1}
\]
we conclude via $s_{k-1}\to\infty$ as $k\to\infty$.


  \section{Open problems} \label{oppro}

  \subsection{Some open problems}
    We want to state some open problems and conjectures.
    Our proofs of the claims above suggest
    a metrical refinement of Theorem~\ref{T01}.

\begin{conjecture}   \label{CCC}
  The set $\Theta$ in Theorem~\ref{T01}
 has Hausdorff dimension at least $2/3$.
\end{conjecture}

   If Conjecture~\ref{CCC} indeed does apply, this would show another strong contrast to the classical Littlewood problem, since counterexamples to \eqref{eq:lw} have been shown to be a set of Hausdorff dimension zero~\cite{ein}.
   See \S~\ref{Co2} for heuristic arguments supporting Conjecture~\ref{CCC}. Besides Theorem~\ref{tth} and Theorem~\ref{t2} give some indication for the stronger bound $1$.
    
    Next we note that our construction strongly suggests that a sequence of
    values $Q_k\to\infty$ for which a strictly positive upper limit in
    Theorem~\ref{T01}
    is attained can be taken of bounded logarithmic growth, i.e. $\log Q_{k+1}/\log Q_k\ll 1$, probably the upper bound $2+o(1)$ as $k\to\infty$ suffices. However our method also definitely requires the reverse estimate $\limsup_{k\to\infty}\log Q_{k+1}/\log Q_k\ge 2$ since $\xi, \zeta$ in our examples have irrationality exponent at least three. This motivates the following problem.

    \begin{problem} \label{P1}
        Do there exist real numbers $\xi, \zeta$ so that for a sequence $Q_k\to\infty$ of growth $\log Q_{k+1}/\log Q_k=1+o(1)$ as $k\to\infty$, we have
        \[
        \lim_{k\to\infty} Q_k \min_{q\in \mathbb{Z},0<q\le Q_k} \Vert q\xi\Vert\cdot \Vert q \zeta\Vert > 0.
        \]
    \end{problem}

    It is easily checked 
    that the classical Littlewood Conjecture~\eqref{eq:lw}
    being false is equivalent to the stronger
    condition $Q_{k+1}/Q_k\ll 1$ in context
    of Problem~\ref{P1}.
    A related problem is to find counterexamples of vectors $(\xi,\zeta)$ whose entries are badly approximable (have bounded partial quotients), or alternatively which are badly approximable as vectors in $\mathbb{R}^2$. 
    We use the notation \eqref{eq:FGH}.

    \begin{problem} \label{P3}
       Do there exist $(\xi,\zeta)\in \mathcal{B}_2(\mathbbm{1})$, i.e. such that
         \[
        \limsup_{Q\to\infty} Q \min_{q\in \mathbb{Z},0<q\le Q} \Vert q\xi\Vert\cdot \Vert q \zeta\Vert  > 0,
        \]
       and additionally
       \begin{itemize}
            \item[(i)]
        $(\xi, \zeta)\in Bad_1\times Bad_1$ (or alternatively at least one of them in $Bad_1$) and/or
        \item[(ii)] $(\xi,\zeta)\in Bad_2$?
         \end{itemize}
    \end{problem}

    A positive answer to (i), (ii) simultaneously in Problem~\ref{P3} would provide a slightly more convincing argument against the classical Littlewood Conjecture~\eqref{eq:lw} than Theorem~\ref{T01}, whereas a negative answer to (i) or (ii) would imply the Littlewood Conjecture.
    Next, in our proof of Theorem~\ref{T01}, it appears pivotal to have only two real variables.
    This motivates

    \begin{problem}[ULC for three variables]
        Do there exist real $\xi_1, \xi_2,\xi_3$ such that 
        \[
        \limsup_{Q\to\infty} Q \min_{q\in \mathbb{Z},0<q\le Q} \Vert q\xi_1\Vert\cdot \Vert q \xi_2\Vert \cdot\Vert q\xi_3\Vert > 0.
        \]
    \end{problem}

This is the special case $m=3, n=1$ of~\cite[Question~1.11]{bfk}.
In this context, we point out that a dual
version of Theorem~\ref{T01} for a single 
linear form in two variables can be derived
from transference methods,
namely that
\[
\sup_{(\xi,\zeta)\in\mathbb{R}^2 }\;\limsup_{Q\to\infty} Q^2\min_{q_i\in \mathbb{Z},0<q_1q_2\le Q} \Vert q_1\xi+q_2\zeta\Vert \ge c
\]
for some effectively computable $c>0$.
We omit computation of a suitable value $c$.
We next ask a presumably difficult question about the ``multiplicative Dirichlet spectrum'' as defined below.

\begin{problem}
    Determine the multiplicative Dirichlet spectrum
    \[
      \left\{ \limsup_{Q\to\infty} Q \min_{q\in \mathbb{Z},0<q\le Q} \Vert q\xi\Vert \cdot \Vert q \zeta\Vert: (\xi,\zeta)\in\mathbb{R}^2\right\}\subseteq [0,1/2].
    \]
\end{problem}

The minimal information we get
from combination of~\cite{bfk} and Theorem~\ref{T01} is that it has at least
two elements one of them being $\{0\}$. Theorem~\ref{T01} and its proof suggest the spectrum contains some interval $[0,c]$, $0<c\le 1/2$. Theorem~\ref{2t} indicates it may be the entire interval $[0,1/2]$. The topological method from~\cite{aw} may be helpful to show that the spectrum is an interval, if true.

Finally we want to ask a question about~\S~\ref{s1.2}.

\begin{problem}
    Does Theorem~\ref{2t} remain true for $\Phi=\mathbbm{1}$? 
\end{problem}

If so, this would also strengthen the subsequent Corollaries~\ref{tth}-\ref{lastc} in a natural way.

\subsection{On Conjecture~\ref{CCC} } \label{Co2}

We finally provide some evidence for Conjecture~\ref{CCC} via some twist in our argument in \S~\ref{lemproof}.
The conjecture is supported by the following heuristic argument: 
For $0<\alpha<\beta$ small enough parameters, let $\Omega\subseteq \mathbb{R}$ be the set of 
real numbers $\xi$ satisfying
\begin{equation} \label{eq:SAT}
\alpha q^{-3}<|\xi-p/q|< \beta q^{-3}
\end{equation}
 for infinitely many rationals $p/q$ with $q$ being prime and not in the exceptional
 set of Theorem~\ref{bkthm}. 
Now the following two hypotheses combined would imply the conjecture:

{\em Hypothesis I}: The set $\Omega$ above 
has the same Hausdorff dimension $2/3$ as the entire set of real numbers with irrationality exponent three (or equivalently as for the set for which we impose only the upper bound in \eqref{eq:SAT}).

{\em Hypothesis II}: A generic number $\xi\in\Omega$ (in particular still a set of the same Hausdorff dimension) induces some $\zeta$ so that
 $(\xi,\zeta)$ is a counterexample to ULC as in Theorem \ref{T01}.
 
 If both hypotheses hold, then
 Conjecture~\ref{CCC} holds by non-increase
 of Hausdorff dimension under projections, since $(\xi,\zeta)\to \xi$ surjectively projects
 a set of counterexamples to ULC onto $\Omega$.

Hypothesis I is suggested by results on exact approximation~\cite{bug, bug2} and results 
 from approximation with denominators from a given set by
 Borosh and Fraenkel~\cite{bofr}. The latter implies that the
 according superset of $\Omega$ where only the upper bound in \eqref{eq:SAT} holds (with prime denominators)
 indeed has
 the same Hausdorff dimension $2/3$ as the entire set of real numbers with irrationality exponent three. Combining the methods is likely to close the gap and prove Hypothesis I.

Hypothesis II is more delicate and requires to modify the line of arguments in \S~\ref{lemproof}. Given $p/q$ a good approximation as in \eqref{eq:SAT} to $\xi\in\Omega$ with $q$ prime, we always 
find many $s$ (to be precise $\gg q^{\sigma}$ many, for $\sigma$ as in Theorem~\ref{bkthm}) so that
$\{ ps/q\} \in F_m$, by just taking $s=up^{-1}\bmod q$ for $u$ any Zaremba $m$ numerator for $q$. We expect that many of these $s$ are prime and of size roughly $q$, as a generic
choice of values for $s$ induces a counterprobability (for primality) of order
$(1-1/\log q)^{q^{\sigma}}\to 0$ 
as $q\to\infty$
(easily seen via 
and $s\asymp q$ and
$s^{\sigma}\gg (\log s)^y$
for arbitrary $y\ge 1$
and the classical formula $\lim_{n\to\infty}(1-1/n)^n=1/e$) for none of them being prime. For each such $s$, 
further
by Lemma~\ref{dadada} it is suggestive
that we find $r$ so that $\{ rq/s\}\in F_m$
as well, as otherwise $q$ would lie in its small set of counterexamples of cardinality $\ll s^{1-\eta}$ for which such $r$ does not exist. So we should expect an infinite subsequence of the well approximating
rationals from \eqref{eq:SAT} to induce $r,s$ so that $p,q,r,s$ satisfy the conditions of Lemma~\ref{lemur} with any $\tau>1$ and $m=50$, and we can conclude
as in \S~\ref{lise}. 

Regarding both hypotheses,
note also that primality
of $q,s$ may not be necessary for the argument. It is just used for direct application of Lemma~\ref{dadada}.
See however Remark~\ref{remark5}.

We point out that
even the following stronger fibered version of Conjecture~\ref{CCC} may apply: 
For any $\xi\in\Omega$, there
is $\zeta$ such that ULC fails for the pair $(\xi,\zeta)$ (with $c>1/188$ again). This would mean the exceptional set of Hypothesis II is empty.

While $2/3$ is a rather moderate conjectural bound,
the best Hausdorff dimension bound we can hope for with our method is $5/3$ 
due to the condition \eqref{eq:SAT} being involved (even if omitting primality condition; in fact this condition applies to  both $\xi, \zeta$, but the Hausdorff dimension of the Cartesian product of such sets could nevertheless possibly be as large as $2/3+1=5/3$). 

{\em Acknowledgements: The author
warmly thanks Igor Shparlinski for providing help with Lemma 4.1. }

\newpage


\begin{thebibliography}{99}

\bibitem{aw} A. Agon, B. Weiss. The Dirichlet spectrum. {\em arXiv:2412.05858}.

\bibitem{badz} D. Badziahin. Computation of the infimum in the Littlewood conjecture. {\em Exp. Math.} 25 (2016), no. 1, 100--105.

  \bibitem{bfk} P. Bandi, R. Fregoli, D. Kleinbock.
  Submanifold-genericity of $\mathbb{R}^d$-actions and
 uniform multiplicative Diophantine approximation. {\em arXiv: 2504.02258}.


  \bibitem{bofr} I. Borosh, A.S. Fraenkel.
A generalization of Jarn\'ik’s theorem on Diophantine approximations.
Nederl. Akad. Wetensch. Proc. Ser. A 75.
{\em Indag. Math.} 34 (1972), 193--201.

  \bibitem{bko} J. Bourgain, A. Kontorovich. On Zaremba's conjecture. {\em Ann. of Math.} (2) 180 (2014), no. 1, 137--196.

 \bibitem{ven} J. Bourgain, E. Lindenstrauss, P. Michel, A. Venkatesh. Some effective results for $\times a\times b$. {\em Ergodic Theory Dynam. Systems} 29 (2009), no. 6, 1705--1722. 

 \bibitem{bug} Y. Bugeaud. Sets of exact approximation order by rational numbers. {\em Math. Ann.} 327 (2003), no. 1, 171--190.

 \bibitem{bug2} Y. Bugeaud. Sets of exact approximation order by rational numbers. II.
{\em Unif. Distrib. Theory} 3 (2008), no. 2, 9--20.


  \bibitem{ein} M. Einsiedler, A. Katok, E. Lindenstrauss. Invariant measures and the set of exceptions to Littlewood's conjecture. {\em Ann. of Math.} 164 (2006), no. 2, 513--560.

  \bibitem{erd} P. Erd\H{o}s. Representations of real numbers as sums and products of Liouville numbers. {\em Michigan
Math. J.} 9 (1962), 59--60.

\bibitem{frokan} D.A. Frolenkov, I.D. Khan. A strengthening of a theorem of Bourgain-Kontorovich II. {\em Mosc. J. Comb. Number Theory} 4 (2014), no. 1, 78--117.

\bibitem{garspa} M.Z. Garaev, I.E. Shparlinski. On some congruences and exponential sums. {\em Finite Fields Appl.} 98 (2024), Paper No. 102451, 17 pp. 



  \bibitem{jarnik} V. Jarn\'ik. Zur metrischen Theorie der diophantischen Approximationen. Prace Mat.-Fiz. 36 (1928/29), 91--106. 
 
   \bibitem{kw} D. Kleinbock, C. Wu. Simultaneously bounded and dense orbits for commuting Cartan actions. {\em arXiv: 2509.05272}.

   \bibitem{maul} R.D. Mauldin, M. Urba\'nski, 
   Conformal iterated function systems with applications to the geometry of
continued fractions. {\em Trans. Amer. Math. Soc.} 351 (1999), no. 12 (1999).
   
   
   Dimensions and measures in infinite iterated function systems, Proc. London Math. Soc. 73 (1996), no. 1, 105--154.

   \bibitem{oxtoby}  J. Oxtoby. Measure and Category. A survey of the analogies between topological and measure spaces.
Second edition. Graduate Texts in Mathematics, 2. Springer-Verlag, New York-Berlin, 1980. x+106
pp.

   \bibitem{Perron} Perron. Lehre von den Kettenbr\"uchen. Textbook Springer 1977.

   \bibitem{icharxiv} J. Schleischitz. Metric results on inhomogeneously singular vectors. {\em arXiv: 2201.1527}.

   \bibitem{schmidt} W.M. Schmidt. On badly approximable numbers and certain games. {\em Trans. A.M.S.} 123 (1966), 27--50.

   \bibitem{tricot} C. Tricot Jr. Two definitions of fractional dimension. {\em Math. Proc. Cambridge Philos. Soc.} 91 (1982),
no. 1, 57--74.

\bibitem{vin} I. M. Vinogradov. An Introduction to the Theory of Numbers (Pergamon Press, 1955).

   \bibitem{zar} S. K. Zaremba. La m\'ethode des “bons treillis” pour le calcul des int\'egrales multiples. In Applications
of number theory to numerical analysis (Proc. Sympos., Univ. Montreal, Montreal, Que., 1971), pages
39–119. Academic Press, New York, 1972. 2.


   \end{thebibliography}
\end{document}